\documentclass{article}

\usepackage{amsmath,amssymb,amsthm,mathrsfs,float}
\usepackage{graphicx}

				\newtheorem{theorem}{Theorem}

                \newtheorem{proposition}[theorem]{Proposition}
\theoremstyle{definition} 
\theoremstyle{definition} 
\theoremstyle{definition} \newtheorem{remark}{Remark}
\theoremstyle{definition}

\begin{document}

\title{An Extension of the Cardioid Distributions on Circle}
\author{Erfan Salavati\\ \\
Department of Mathematics and Computer Science,\\
Amirkabir University of Technology (Tehran Polytechnic),\\
P.O. Box 15875-4413, Tehran, Iran.}
\date{}

\maketitle

\begin{abstract}
A new family of distributions on the circle is introduced which are a generalization of the Cardioid distributions. The elementary properties such as mean, variance and the characteristic function are computed.  The distribution is either unimodal or bimodal. The modes are computed. The symmetry of the distribution is characterized. The parameters are shown to be canonic (i.e. uniquely determined by the distribution).
We also show that this new family is a subset of distributions whose Fourier series has degree at most 2 and study the implications of this property.
\end{abstract}

Key Words: Circular Distributions, Cardioid Distribution, Von Mises Distribution

MSC: 62E10, 43A05, 43A25

\section{Introduction}
Circular data arise in many natural phenomenon. The main two categories of such data are physical directions and periodical time records. Wind direction and the direction of migrating birds are two examples of physical directions. The arrival times measured by clock and the date of specific observations (in a year) are two examples of periodical time records. Circular data are usually measured and represented by degrees or radians. Axial data are obtained from circular data by doubling them.

Circular distributions are an important tool in analysing and inference of circular data. One of the elementary and useful circular distributions is the Cardioid distribution. This distribution has two parameters $\mu\in [0,2\pi)$ and $\rho$ with $|\rho|<\frac{1}{2}$ and is denoted by $C(\mu,\rho)$ and has the density function
\[ f(\theta) = \frac{1}{2\pi} (1+2\rho \cos(\theta-\mu)). \]
Its distribution is symmetric around $\mu$ and is unimodal with a mode at $\mu$ and an anti-mode at $\mu+\pi$. This family of distributions is closed under convolution (summation of independent copies) and mixtures. For further properties see~\cite{Mardia}, page 45.

Another popular and useful circular distribution is Von Mises distribution. This distribution has two parameters $\mu\in [0,2\pi)$ and $\kappa>0$ and is denoted by $VM(\mu,\kappa)$ and has the probability density function
\[ g(\theta) = \frac{1}{2\pi I_0(\kappa)} e^{\kappa \cos(\theta-\mu)}. \]
where
\[ I_0(\kappa) = \frac{1}{2\pi} \int_0^{2\pi} e^{\kappa \cos(\theta)}d\theta. \]
The parameter $\mu$ is the mean direction and $\kappa$ is called the concentration parameter. This distribution is unimodal and symmetric about $\theta=\mu$. $\mu$ is the mode and $\mu+\pi$ is the anti-mode. For further information on Von Mises distributions see~\cite{Mardia}, page 36.

We should also mention a recently introduced extension of the von Mises distribution in~\cite{Kato_Jones}. Their family of distibutions is obtained by applying a M\"obius transformation on the von Mises distribution. This family also contains the wrapped Cauchy distribution.

In this article, we introduce a new family of circular distributions which generalize the cardioid distributions. We call this new distribution the quadratic cardioid distribution because it's density function has a second order triangular representation.

In the next section, after introducing the distribution, we will compute some elementary statistics and the characteristic function. We will also discuss the inverse problem of computing the parameters given the distribution. At last we discuss the mixtures and convolutions of such distributions.

\section{Mathematical Definition and Properties}
Let $\mu_1\le \mu_2 \in [0,2\pi)$ and $r_1,r_2\ge 0$. By a quadratic cardioid distribution, denoted by $QC(\mu_1,\mu_2,r_1,r_2)$, we mean a distribution with probability density function
\begin{multline*}
 f(\theta;\mu_1,\mu_2,r_1,r_2) =\\
 \frac{1}{I(r_1,r_2)} \Big(1+r_1^2+r_2^2+2r_1 \cos(\theta-\mu_1) + 2r_2 \cos(\theta-\mu_2)\\
  + 2r_1 r_2 \cos (2\theta-\mu_1-\mu_2)\Big)
 \end{multline*}
where $I(r_1,r_2)=2\pi (1+r_1^2+r_2^2)$.

Note that
\[ f(\theta;\mu_1,\mu_2,r_1,r_2)  = \frac{1}{I(r_1,r_2)} \left| 1+r_1 e^{i(\theta-\mu_1)} + r_2 e^{-i(\theta-\mu_2)} \right|^2 \]
and hence we have $f \ge 0$.

\subsection{Special Cases}
If $r_1=r_2=0$ this would be the uniform distribution on the circle and if $r_2=0$, this would be the ordinary cardioid distribution $C(\mu_1,r_1)$ and similarly for $r_1$.

Another interesting special case is provided in the following proposition.

\subsection{Elementary Statistics}
Expectation of the QC distribution is $\frac{2\pi}{I} (r_1 e^{i\mu_1} +r_2 e^{i\mu_2})$ and hence the mean direction is $\bar{\theta} = Arg (r_1 e^{i\mu_1} +r_2 e^{i\mu_2})$. Another easy calculation gives that the mean resultant length, $\bar{R}$ equals $\sqrt{r_1^2+r_2^2+2r_1r_2 \cos(\mu_1-\mu_2)}$.

For the median we should have,
\[ 0 = \int_\phi^{\phi+\pi} f(\theta) d\theta = \frac{1}{2} + 4 r_1 \sin(\phi - \mu_1) + 4 r_2 \sin(\phi - \mu_2) \]
Hence which is equivalent to $r_1 \sin(\phi - \mu_1) +  r_2 \sin(\phi - \mu_2)=0$. Solving this equation gives,
\[ \sin (\phi-\mu_1) = \pm \frac{r_2\sin(\mu_2-\mu_1)}{\bar{R}^2}\]
\[ \sin (\phi-\mu_2) = \pm \frac{r_1\sin(\mu_1-\mu_2)}{\bar{R}^2}\]
Above equations have two solutions which differ by $\pi$, the median is the one which $r_1 \cos(\phi - \mu_1) +  r_2 \cos(\phi - \mu_2)\le 0$

\subsection{Characteristic function}
A simple computation shows that the characteristic function of $QC(\mu_1,\mu_2,r_2,r_2)$ is
\[ \int_0^{2\pi} e^{in\theta} f(\theta) d\theta = \left\{ \begin{array}{ll}
1& n=0\\
\frac{2\pi}{I} (r_1 e^{i\mu_1} +r_2 e^{i\mu_2})& n=1\\
\frac{2\pi r_1 r_2}{I} e^{2i(\mu_1+\mu_2)} & n=2\\
0& n\ne 0,1,2
\end{array} \right. \]

\begin{remark}
	For small values of $r_1,r_2$, the QC distribution is an approximation of the generalised Von Mises distribution which is introduced and studied in~\cite{Yfantis}.
\end{remark}

\section{The Shape of the Distribution}
For different values of the parameters, the QC distribution can be symmetric or asymmetric, unimodal or bimodal.

\subsection{Conditions of Symmetry}
\begin{proposition}
The QC distribution is symmetric if and only if at least one of the following holds:
\begin{description}
	\item[(i)] $r_1=r_2$.
	\item[(ii)] $\mu_1=\mu_2$.
	\item[(iii)] at least one of $r_1$ and $r_2$ are zero.
\end{description}
In case (i), the distribution is symmetric about $\frac{\mu_1+\mu_2}{2}$,
in case (ii), the distribution is symmetric about $\mu_1=\mu_2$ and in case (iii), the distribution is symmetric about $mu_1$ (or $\mu_2$).
\end{proposition}

\begin{proof}
	The if part is straightforward. For the only if part, assume the distribution is symmetric about $\theta_0$ (and hence about $\theta_0+\pi$). Denote the density function by $f(\theta)$ and assume that both of $r_1$ and $r_2$ are nonzero. We have,
\[ f^{(2n)}(\theta) = \frac{1}{I(r_1,r_2)}(-1)^n (2r_1 \cos(\theta-\mu_1) + 2r_2 \cos(\theta-\mu_2) + 2^{2n+1} r_1 r_2 \cos(2\theta-\mu_1-\mu_2) \]
Hence we have,
\[ \lim_{n\to\infty} f^{(2n)}(\theta) 2^{-2n} = r_1 r_2 \cos(2\theta-\mu_1-\mu_2) \]
since $f$ is symmetric about $\theta_0$ then so is $f^{(2n)}$ and hence $\cos(2\theta-\mu_1-\mu_2)$. This implies that $\theta_0=\mu_1+\mu_2$ (or $\mu_1+\mu_2+\pi$ which makes no difference in the remainder of the proof). And also we find that $2r_1 \cos(\theta-\mu_1) + 2r_2 \cos(\theta-\mu_2)$ is also symmetric about $\theta_0$. This, combined with $\mu_1=\mu_2$ implies that $r_1=r_2$.

\end{proof}

\subsection{Modes and Unimodality}
In general, modes and anti-modes are the roots of the following function,
\begin{equation}\label{eq_mode}
	f^\prime(\theta) = r_1 \sin (\theta-\mu_1)+ r_2 \sin(\theta-\mu_2) + 2 r_1 r_2 \sin(2\theta - \mu_1 - \mu_2)
\end{equation}
This equation can be solved analytically for general parameters. In fact, one can expand the expression in terms of $t=\tan(\frac{\theta}{2})$ and find a 4th degree equation which can be solved analytically using Ferrari's equation.

 In the special case $\mu_1=\mu_2=\mu$, the distribution is a cardioid distribution and is symmetric about $\mu$. In this case, two solutions of\eqref{eq_mode} are $\theta=\mu$ and $\theta=\mu+\pi$. The former is always a mode and the latter is a mode if and only if $4r_1r_2 > r_1+r_2$. If this inequality holds then there are two other solutions for\eqref{eq_mode} which are the solutions of $\cos(\theta-\mu_1)=-\frac{r_1+r_2}{4r_1 r_2}$ and both are anti-modes.

\subsection{Graphs of Density}
Figure~\ref{fig:QC_density} shows the different shapes of the density of a quadratic Cardioid distribution for different values of parameters.

Columns from left to right correspond respectively to $r_1=\frac{1}{2},1,\frac{3}{2}$ and $\mu_1=1,\frac{3}{2},2$ and rows from top to bottom correspond respectively $r_2=\frac{1}{2},1,\frac{3}{2}$ and $\mu_2=\frac{1}{2},1,\frac{3}{2}$.

\begin{figure}[h]
\label{fig:QC_density}
	\begin{center}
		\includegraphics[scale=0.28]{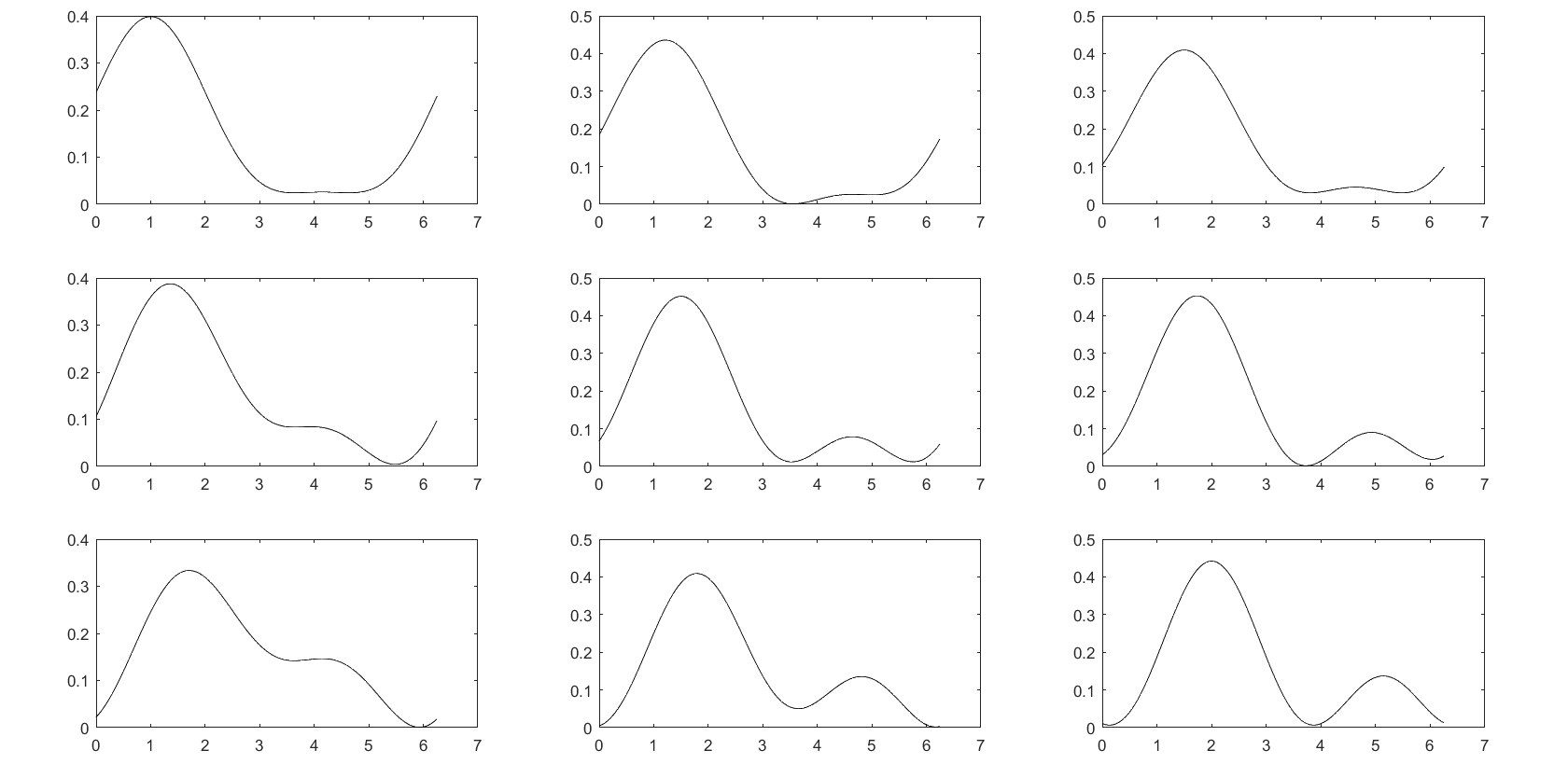}
	\end{center}
\end{figure}



\section{Uniqueness of Parameters}

\begin{theorem}
Parameters of a QC distribition are uniquely determined by the distribution.
\end{theorem}

\begin{proof}
	Assume that $QC(\mu_1,\mu_2,r_2,r_2)\approx QC(\mu_1^\prime,\mu_2\prime,r_1^\prime,r_2^\prime)$.
	Define
	\[ g(z) = r_1 e^{-i\mu_1} z^2 + z + r_2 e^{i\mu_2} \]
	\[ h(z) = r_1^\prime e^{-i\mu_1^\prime} z^2 + z + r_2^\prime e^{i\mu_2^\prime} \]
It follows from the assumption that for $z=e^{i\theta}$, $\left|\frac{g(z)}{h(z)}\right|$ is a constant $c$. Hence the function $\frac{g(z)}{ch(z)}$ maps the unit circle $|z|=1$ into itself.

Now it follows from the Schwartz lemma that $\frac{g(z)}{ch(z)}$ is a product of two Mobius functions $e^{i\alpha}\frac{z-a}{1-\bar{a}z}$ and $e^{i\beta}\frac{z-b}{1-\bar{b}z}$. Substituting implies that $(\mu_1,\mu_2,r_1,r_2) = (\mu_1^\prime,\mu_2\prime,r_1^\prime,r_2^\prime)$
\end{proof}

\section{QC as a parametrization of second order positive definite Fourier series}

Let $\mathcal{M}$, $\mathcal{M}^+$ and $\mathcal{M}^{\pi}$ be respectively the set of all signed Borel measures, positive Borel measures and probability Borel measures on $\mathbb{S}^1$.

We define the spaces $\mathcal{T}_N$ to be the set of all positive signed Borel measures on $\mathbb{S}^1$ whose Fourier series has degree at most $N$. In other words,
\[ \mathcal{T}_N = \{ \mu \in \mathcal{M}: \hat{\mu}(n) = 0\quad \forall n\notin \{-N,\ldots,0,\ldots,N\} \} \]
and define
\[ \mathcal{T}^+_N = \mathcal{T}_N \cap \mathcal{M}^+,\quad \mathcal{T}^\pi_N = \mathcal{T}_N \cap \mathcal{M}^\pi\]

\begin{proposition}
$\mathcal{T}^+_N$ is closed under linear combination and convolution. $\mathcal{T}^\pi_N$ is closed under mixture and convolution.
\end{proposition}
\begin{proof}
Since Fourier transform is a linear transform, hence closedness under linear combinations is trivial. On the other hand, convolution of measures translates in to multiplication of Fourier transforms and hence closedness under convolution also follows. The statement for $\mathcal{T}^\pi_N$ follows easily from the fact that $\mathcal{M}^\pi$ is closed under mixture and convolution.
\end{proof}

Let $\mathcal{C}$ and $\mathcal{QC}$ be respectively the space of all Cardioid and Quasi-Cardioid distributions.

A simple observation shows the relation between Cardioid distributions and $\mathcal{T}^\pi_N$ spaces.

\begin{proposition}
We have $\mathcal{C}=\mathcal{T}^\pi_1$
\end{proposition}

One of the interesting properties of QC distributions is stated in the following proposition.

\begin{proposition}
We have $\mathcal{QC} \subset \mathcal{T}^\pi_2$ and that $\mathcal{QC}$.
\end{proposition}

\begin{remark}
    Indeed it may be the case that $\mathcal{QC} = \mathcal{T}^\pi_2$ but we have not yet succeeded in proving or disproving it.
\end{remark}

The previous proposition implies that the QC distributions are actually parametrizing the space $\mathcal{T}^\pi_2$ at least partially. A good question is that to what extent is this parametrization complete? In other words, what portion of $\mathcal{T}^\pi_2$ is covered by this parametrization?

For this reason we need to characterize the elements of $\mathcal{T}^\pi_2$ in terms of their Fourier series. We use a well-known theorem due to Bochner which provides a necessary and sufficient condition for a function to be the characteristic function of a probability distribution. The Bochner theorem holds in general for all probability measures on dual group of Abelian groups, but we state here only the special case of probability measures on $\mathbb{S}^1$.

\begin{theorem}[Bochner]
A function $\hat{f}:\mathbb{Z}\to \mathbb{C}$ is the Fourier series of a positive Borel measure on $\mathbb{S}^1$ if and only if $\hat{f}(0)=2\pi$ and the following matrix is positive definite:
\[\left[ {\begin{array}{*{20}{c}}
  {\hat f(0)}&{\hat f(1)}&{\hat f(2)}& \cdots & \cdots  \\ 
  {\hat f( - 1)}&{\hat f(0)}&{\hat f(1)}&{\hat f(2)}& \cdots  \\ 
  {\hat f( - 2)}&{\hat f( - 1)}&{\hat f(0)}&{\hat f(1)}& \ddots  \\ 
   \vdots &{\hat f( - 2)}&{\hat f( - 1)}& \ddots & \ddots  \\ 
   \vdots & \vdots & \ddots & \ddots & \ddots  
\end{array}} \right]\]
\end{theorem}

Let $\hat{\nu}$ be the Fourier series of an element $\nu$ in $\mathcal{T}^\pi_2$. We assume $\hat{\nu}(1)=2\pi c_1$ and $\hat{\nu}(2)=2\pi c_2$. It follows that $\hat{\nu}(-1)=2\pi \bar{c_1}$ and $\hat{\nu}(-2)=2\pi \bar{c_2}$.

Applying the Bochner theorem now implies,
\begin{proposition}
$\nu \in \mathcal{T}^\pi_2$ if and only if
$\hat{\nu}(0)=2\pi$, $\hat{\nu}(1)=2\pi c_1$ and $\hat{\nu}(2)=2\pi c_2$ and the following five diagonal matrix is positive definite:
\[\left[ {\begin{array}{*{20}{c}}
  1&{{c_1}}&{{c_2}}&0& \cdots  \\ 
  {{{\bar c}_1}}&1&{{c_1}}&{{c_2}}& \cdots  \\ 
  {{{\bar c}_2}}&{{{\bar c}_1}}&1&{{c_1}}& \ddots  \\ 
  0&{{{\bar c}_2}}&{{{\bar c}_1}}& \ddots & \ddots  \\ 
   \vdots & \vdots & \ddots & \ddots & \ddots  
\end{array}} \right]\]

Applying the determinant criterion for positive definiteness, gives us several necessary conditions:
\[- c_{1} \overline{c_{1}} + 1 >0 \]
\[ c_{1}^{2} \overline{c_{2}} - 2 c_{1} \overline{c_{1}} + c_{2} \overline{c_{1}}^{2} - c_{2} \overline{c_{2}} + 1>0 \]
\[ c_{1}^{2} \overline{c_{1}}^{2} + 2 c_{1}^{2} \overline{c_{2}} - 2 c_{1} c_{2} \overline{c_{1}} \overline{c_{2}} - 3 c_{1} \overline{c_{1}} + c_{2}^{2} \overline{c_{2}}^{2} + 2 c_{2} \overline{c_{1}}^{2} - 2 c_{2} \overline{c_{2}} + 1 >0\]
\end{proposition}

\end{document}